\documentclass[11pt,twoside]{amsart}
\usepackage{amsmath, amsthm, amscd, amsfonts, amssymb, graphicx, color}
\usepackage[bookmarksnumbered, colorlinks, plainpages]{hyperref}

\newtheorem{thm}{Theorem}[section]
\newtheorem{cor}[thm]{Corollary}
\newtheorem{lem}[thm]{Lemma}
\newtheorem{prop}[thm]{Proposition}

\numberwithin{equation}{section}
\def\pn{\par\noindent}

\begin{document}

\leftline{ \scriptsize \it Bulletin of the Iranian Mathematical
Society  Vol. {\bf\rm XX} No. X {\rm(}201X{\rm)}, pp XX-XX.}

\vspace{1.3 cm}

\title{On  solubility  of groups with finitely many centralizers}
\author{Mohammad Zarrin}

\thanks{{\scriptsize
\hskip -0.4 true cm MSC(2010): Primary: 20D99; Secondary: 20E07.
\newline Keywords: Centralizer, n-centralizer group, simple group.\\
Received: 25 December 2011, Accepted: 18 April June 2012.\\
\newline\indent{\scriptsize $\copyright$ 2011 Iranian Mathematical
Society}}}

\maketitle

\begin{center}
Communicated by\;
\end{center}

\begin{abstract}  For any group $G$, let $\mathcal{C}(G)$ denote the set of
centralizers of $G$. We say that a group $G$ has $n~centralizers$
($G$ is a $\mathcal{C}_n$-group) if $|\mathcal{C}(G)| = n$. In
this note, we prove that every finite $\mathcal{C}_n$-group with
$n\leq 21$ is soluble and this estimate is sharp. Moreover, we
prove that every finite $\mathcal{C}_n$-group with $|G|<
\frac{30n+15}{19}$ is non-nilpotent soluble. This result gives a
partial answer to a conjecture raised by A. Ashrafi in 2000.
\end{abstract}

\vskip 0.2 true cm


\pagestyle{myheadings}
\markboth{\rightline {\scriptsize  Zarrin}}
         {\leftline{\scriptsize On  solubility  of groups with finitely many centralizers}}

\bigskip
\bigskip


\section{\bf Introduction}
For any group $G$, let $\mathcal{C}(G)$ denote the set of
centralizers of $G$. We say that a group $G$ has $n~centralizers$
($G\in \mathcal{C}_n$, or $G$ is a $\mathcal{C}_n$-group) if
$|\mathcal{C}(G)| = n$. Also we say that $G$ has a finite number
of centralizers, written $G\in \mathcal{C}$, if $G\in
\mathcal{C}_n$ for some $n\in \mathbb{N}$. Indeed
$\mathcal{C}=\bigcup_{i\geq 1}\mathcal{C}_i$. It is clear that a
group is a $\mathcal{C}_1$-group if and only if it is abelian.
 Belcastro and Sherman in~\cite{Be}, showed that there is no
finite $\mathcal{C}_n$-group for $n\in \{2, 3\}$ (while Ashrafi
in~\cite{Ash1}, showed that, for any positive integer $n\neq 2,
3$, there exists a finite group $G$ such that $|\mathcal{C}(G)| =
n$). Also they characterized all finite $\mathcal{C}_n$-groups for
$n\in \{4, 5\}$. Tota (see Appendix of \cite{tot}) proved that
every arbitrary $\mathcal{C}_4$-group is soluble. The author in
\cite{zar1} showed that the derived length of a soluble
$C_n$-group (not necessarily finite) is $\leq n$. For more details
concerning $\mathcal{C}_n$-groups see \cite{abd,Ash1, Ash2,
Ash3,zar1,zar2}. In this paper, we obtain a solubility criteria
for $C_n$-groups in terms of $|G|$ and $n$.

Our main results are:\\

 \noindent{\bf Theorem A.}
 Let $G$ be a finite $\mathcal{C}_n$-group with $n\leq 21$, then $G$
 is soluble. The alternating group of degree 5 has 22
centralizers.\\

\noindent{\bf Theorem B.}
If $G$ is a finite  $\mathcal{C}_n$-group, then the following hold:\\
 (1)\;  $|G|< 2n$, then $G$ is a non-nilpotent
group.\\
(2)\;  $|G|< \frac{30n+15}{19}$, then $G$ is a non-nilpotent
soluble group.\\

Let $G$ be a finite $\mathcal{C}_{n}$-group.  In \cite{Be},
Belcastro and Sherman raised the question whether or not there
exists a finite $\mathcal{C}_{n}$-group $G$ other than $Q_8$ and
$D_{2p}$ ($p$ is a prime) such that $|G|\leq 2n$. Ashrafi in
\cite{Ash1} showed that there are several counterexamples for this
question and then Ashrafi raised the  following conjecture
(conjecture 2.4): If $|G|\leq 3n/2$, then $G$ is isomorphic to
$S_3, S_3\times S_3$, or a dihedral group of order $10$. Now by
Theorem A, we can obtain that if $|G|\leq 3n/2$, then $G$ is
soluble. Therefore Theorem A give a partial answer to the
conjecture put forward by Ashrafi.


\section{\bf {\bf \em{\bf Proofs}}}
\vskip 0.4 true cm

Let $n > 0$ be an integer and $\mathcal{X}$ be a class of groups.
We say that a group $G$ satisfies the condition $(\mathcal{X}, n)$
($G$ is a $(\mathcal{X}, n)$-group) whenever in every subset with
$n+1$ elements of $G$ there exist distinct elements $x, y$ such
that $\langle x, y\rangle$ is in $\mathcal{X}$. Let $\mathcal{N}$
and $\mathcal{A}$ be the classes of nilpotent groups and abelian
groups, respectively.  Indeed, in a group satisfying the condition
$(\mathcal{A}, n)$, the largest set of non-commuting elements (or
the largest set of elements in which no two generate
an abelian subgroup) has size at most $n$.\\
Here we give an interesting relation between groups that have n
centralizers and groups that satisfy the condition $(\mathcal{A},
n-1)$.

\begin{prop}\label{lo}
Let $n$ be a positive integer and $G$ be a $\mathcal{C}_n$-group
{\rm(}not necessarily finite{\rm)}. Then $G$
 satisfies the condition
 $(\mathcal{A},n-1)$.
 \end{prop}


\begin{proof}
Suppose, for a contradiction,  that $G$
 does not satisfy the condition
 $(\mathcal{A},n-1)$. Therefore there exists a subset $X=\{a_1,a_2,\cdots, a_{n}\}$ of $G$ such that
 $\langle a_i, a_j\rangle$ is not abelian, for every $1\leq i\not =j\leq n$.  This
 follows that $C_G(a_i)\neq C_G(a_j)$ for every $1\leq i\not =j\leq n$. Now since $C_G(e)=G$, where $e$ is the trivial element of $G$,
  we get  $n=|\mathcal{C}(G)|\geq n+1$, which is impossible.
 \end{proof}


  Note that by easy computation we can see that the symmetric group of degree 4,
  $S_4$, satisfies  the condition $(\mathcal{A},10)$, but $S_4$ is
  not a $\mathcal{C}_{11}$-group
  (in fact, $S_4$ is a $\mathcal{C}_{14}$-group). That is, the converse of the above
  Proposition is not true.\\

We can now deduce Theorem A.\\

\noindent{\bf{Proof of Theorem A.}} Clearly every group satisfies
the condition $(\mathcal{A}, n)$ also satisfies the condition
$(\mathcal{N}, n)$. Thus, by Proposition \ref{lo}, $G$ satisfies
the condition $(\mathcal{N}, n)$ for some $n\leq 20$. Now this
statement follows from the main result of \cite{en}. By easy
computation we can obtain that the alternating group of degree 5,
has 22 centralizers.


Note that Ashrafi and Taeri in \cite{Ash3}, proved that, if $G$ is
a finite simple group and $|\mathcal{C}(G)| = 22$, then $G\cong
A_5$. Then they, by this result, claimed that, if $G$ is a finite
group and $|\mathcal{C}(G)|\leq 21$, then $G$ is soluble.
Therefore, in view of Theorem A, we gave positive answer to their
claim.


Tota in \cite[Theorem 6.2]{tot}) showed that a group $G$ belongs
to $\mathcal{C}$ if and only if it is center-by-finite. Therefore
it is natural problem to obtain bounds for $|G:Z(G)|$ in terms of
$n$.


\begin{thm}\label{tb}
There is some constant $c\in\mathbf{R}_{>0}$ such that for any
$\mathcal{C}_{n}$-group $G$
$$n\leq |G:Z(G)|\leq c^{n-1}.$$
\end{thm}
\begin{proof}
First, by the main result of \cite{pyb} and Proposition \ref{lo}
we have $|G:Z(G)|\leq c^{n-1},$ for some constant $c$. To remain
the prove we may assume that $Z(G)\neq 1$. Since elements in the
same coset modulo $Z(G)$ have the same centralizer, it follows
that $n\leq |G:Z(G)|$.
\end{proof}


For the proof of Theorem B, we need the following lemma.
\begin{lem}\label{li}
 Let $G$ be a finite $\mathcal{C}_{n}$-group. Then
 $$n\leq\frac{|G|+|I(G)|}{2},$$ where $I(G)=\{a\in G\mid a^2=1\}=\{a\in G\mid a=a^{-1}\}$.
\end{lem}
\begin{proof}
Since $C_G(a)=C_G(a^{-1})$, we can obtain that
$$n\leq |I(G)|+|\frac{G-I(G)}{2}|\leq \frac{|G|+|I(G)|}{2},$$
as wanted.
\end{proof}


\begin{cor}
 Let $G$ be a finite simple $\mathcal{C}_{n}$-group. Then $3n/2< |G|.$
\end{cor}
\begin{proof}
It is well known that for every simple group we have $I(G)<|G|/3$.
Now  the result follows from Lemma \ref{li}.
\end{proof}


Here we show that a semi-simple $\mathcal{C}_{n}$-group has order
bounded by a function of $n$. (Recall that a group $G$ is
semi-simple if $G$ has no non-trivial normal abelian subgroups.)

\begin{prop}
Let $G$ be a semi-simple $\mathcal{C}_{n}$-group. Then $G$ is
finite and $|G| \leq (n-1)!$.
\end{prop}
 \begin{proof}
 The group $G$ acts on the set $A:=\{C_G(x) ~|~ a \in G \setminus Z(G)\}$ by
conjugation. By assumption $|A|=n-1$. Put $B=\bigcap_{x\in
G}N_G(C_G(x))$. The subgroup $B$ is the kernel of this action and
so $$G/B\hookrightarrow S_{n-1}. \eqno{(*)}$$ By definition of
$B$, the centralizer $C_G(a)$ is normal in $B$ for any element
$a\in G$. Therefore $a^{-1}a^b \in C_G(a)$ for any two elements
$a, b\in  B$. So $B$ is a 2-Engel group (see \cite{Kap}). Now it
is well known that $B$ is a nilpotent group of class at most 3.
Now as $G$ is a semi-simple group, we can obtain that $B=1$. It
follows from (*) that $G$ is a finite group and $|G|\leq (n-1)!$,
as wanted.
\end{proof}


 We need the following result, for the proof of Theorem B.
\begin{thm}{\rm(Potter, 1988)}\label{tp}
Suppose $G$ admits an automorphism which inverts more than
$4|G|/15$ elements. Then $G$ is soluble.
\end{thm}


\noindent{\bf{Proof of Theorem B.}} (1).\; Suppose, for a
contradiction, that $G$ is nilpotent group, so in particular
$Z(G)\neq 1$. Now it follows from Theorem \ref{tb} that $2n\leq
|G|$, which is contradiction.\\
(2).\; From part (1) we obtain that $G$ is not nilpotent.
 Since $|G|< \frac{30n+15}{19}$ and so $2n> \frac{19|G|-15}{15}$,
   Lemma \ref{li} implies that
$$|I(G)|\geq 2n-|G|> \frac{4|G|}{15}-1.$$ On the other hand, since
 $I(G)$ is the set of all elements of $G$ that inverted by
the identity automorphism, Theorem \ref{tp} completes the proof.


\vskip 0.4 true cm

\begin{center}{\textbf{Acknowledgments}}
\end{center}
 This
research was supported by
University of Kurdistan. \\ \\
\vskip 0.4 true cm



\bigskip
\bigskip


{\footnotesize \pn{\bf}\; \\ {Department of Mathematics},
{University
of Kurdistan, P.O.Box 416,} {Sanandaj, Iran}\\
{\tt Email: m.zarrin@uok.ac.ir}\\

\begin{thebibliography}{20}

\bibitem{abd} A. Abdollahi, S. M. Jafarian Amiri and A. Mohammadi
Hassanabadi, Groups with specific number of centralizers, {\em
Houston Journal of Mathematics} {\bf 33} (2007) No. 1, pp. 43-57.
\bibitem{Ash1} A. R. Ashrafi, On finite groups with a given number of centralizers,
{\em Algebra Colloq.} {\bf 7} (2000) 139--146.
\bibitem{Ash2} A. R. Ashrafi, Counting the centralizers of some finite groups, {\em Korean J. Comput. Appl. Math.} {\bf
7} (2000) 115--124.
\bibitem{Ash3} A. R. Ashrafi and B. Taeri, On finite groups with a
certain number of centralizers, {\em J. Appl. Math. Computing}
{\bf 17} (2005) 217--227.
\bibitem{Be} S. M. Belcastro and G. J. Sherman, Counting centralizers in finite groups, {\em Math. Mag.} {\bf 5}
(1994) 111--114.
\bibitem{en} G. Endimioni,
Groupes finis satisfaisant la condition $(N, n)$, {\em C. R. Acad.
Sci. Paris Ser. I} {\bf 319} (1994) 1245--1247.
\bibitem{Kap} W. P. Kappe, Die A-Norm einer
Gruppe, {\em Illinois J. Math.} {\bf 5} (1961) 187--197.
\bibitem{pot} W. M. Potter, Nonsolvable
groups with an automorphism inverting many elements, {\em Archiv
der Mathematik (Basel)} {\bf 50} (1988) 292--299.
\bibitem{pyb} L. Pyber, The number of pairwise non-commuting elements and the
index of the centre in a finite group, {\em J. London Math. Soc.}
(2) {\bf 35} (1987) 287--295.
\bibitem{tot} M. Tota, Groups with a Finite Number of Normalizer
Subgroups, {\em Comm. Algebra} {\bf 32} (2004) 4667--4674.
\bibitem{zar1} M. Zarrin, Criteria for the  solubility  of finite groups by its centralizers,  {\em Arch.
Math.} {\rm 96} (2011) 225--226.
\bibitem{zar2} M. Zarrin, On element-centralizers in finite groups, {\em Arch. Math.} {\bf
93} (2009) 497--503.

\end{thebibliography}
\end{document}